\documentclass{article}
\usepackage[total={6.5in, 9in}]{geometry}
\usepackage{amsfonts}
\usepackage[utf8]{inputenc}
\usepackage{indentfirst}
\usepackage[english]{babel}
\usepackage{enumerate}   

\newcommand{\lam}{\mathcal{L}}

\usepackage{hyperref} %

\newcommand\blfootnote[1]{
  \begingroup
  \renewcommand\thefootnote{}\footnote{#1}
  \addtocounter{footnote}{-1}
  \endgroup
}
\usepackage{amsmath}
\usepackage{amsthm}
\usepackage{thmtools}
\usepackage[capitalise,noabbrev]{cleveref}

\newtheorem{thm}{Theorem}[section]
\newtheorem{theorem}[thm]{Theorem}
\newtheorem{lemma}[thm]{Lemma}
\newtheorem{proposition}[thm]{Proposition}
\newtheorem{cor}[thm]{Corollary}
\newtheorem{prop}[thm]{Proposition}
\theoremstyle{definition}
\newtheorem{definition}[thm]{Definition}
\newtheorem{remark}[thm]{Remark}

\newtheorem{conj}[thm]{Conjecture}
\newtheorem{ex}[thm]{Example}

\usepackage{caption}  %
\usepackage{graphicx}
\title{Finite Dynamical Laminations}
\author{Forrest M. Hilton}

\begin{document}

\maketitle
\blfootnote{This work was supported by a UAB Honors College Presidential Honors Fellowship, and mostly written in 2024. The author's research advisor was Dr. John C. Mayer. Some revisions were made under the advice of Dr. Lex G. Oversteegen.}
\begin{abstract}
	We develop several combinatorial notions about laminations, some with clear implications for parameter space. We introduce a simplified class of laminations called finite dynamical laminations (FDL). In order to count FDL, we introduce sibling portraits, of which we provide a comprehensive counting theorem. We provide a characterization of which periodic polygons appear in invariant laminations. We introduce the pullback tree. The base of the pullback tree is a set of laminations, and we show that those laminations are proper and invariant and all laminations in the base of the pullback tree correspond to a polynomial. We define the generational FDL graph, and it provides combinatorial information about polynomial parameter space.
\end{abstract}
\section{Introduction}

A lamination is a model of a polynomial. Here a polynomial is a complex polynomial. We provide some exposition now, but most background is withheld until needed.

\begin{definition}
	A \emph{lamination} is a closed set of chords of the unit disk that do not cross each other except at the endpoints that additionally includes every point of the circle. A \emph{leaf} is a chord in a lamination joining 2 different points, while a \emph{degenerate leaf} is a point of the circle.
\end{definition}

Our convention is that a degenerate leaf is not a leaf. The equivalence relation of a lamination is the minimal equivalence relation such that both endpoints of every leaf are in the same equivalence class. Many Julia sets of polynomials are homeomorphic to the quotient of the circle with a lamination. Moreover, it is long known how to find laminations that serve as dynamical models of a Julia set.

The angle on the circle is measured in turns, thus $\theta \in [0,1)$. Given that we want to model degree $d$ polynomial, the function $\sigma_d(\theta) = \theta d$ mod $1$ defines the dynamics on the circle. Our notation for angles omits the decimal/radix point as implied since all numbers are less than one. We write numbers in base $d$, and the repeating part of a rational number is offset by ``\_'' instead of the traditional overline. Thus, $0\_001 = 0.0\overline{001}_d$.

We often apply $\sigma_d$ to leaves or arcs of the circle. The leaf from $a$ to $b$ is denoted $\overline{ab}$. The motion of a leaf is according to its endpoints: $\sigma_d(\overline{ab})= \overline{\sigma_d(a)\sigma_d(b)}$. For an arc of the circle, $A$, we define $\sigma_d(A)$ as the arc containing $\sigma_d(\theta)$for all $ \theta \in A$. For regions of the disk bounded by arcs and leaves, the image is the region bounded by the image of each part of the boundary. Since the Julia set is $d$ to 1 invariant under the polynomial, we model it with laminations that are similarly invariant under $\sigma_d$. We provide the most modern notion of an invariant lamination \cite{blokh2012laminationslanguageleaves}:

\begin{definition}
	A lamination $\lam$ is {\em sibling $d$-invariant} or simply {\em $d$-invariant} if:

	\begin{enumerate}

		\item for each $\ell\in\mathcal{L}$ either
		      $\sigma_d(\ell)\in\mathcal{L}$ or $\sigma_d(\ell)$ is a point
		      in $\mathbb{S}$,

		\item for each $\ell\in\mathcal{L}$ there exists a leaf
		      $\ell'\in\mathcal{L}$  such that $\sigma_d(\ell')=\ell$,

		\item \label{disjoint} for each $\ell\in\mathcal{L}$ such that
		      $\sigma_d(\ell)$ is not a point, there exist $\mathbf d$ {\bf
				      disjoint} leaves $\ell_1, \dots, \ell_d$ in $\mathcal{L}$ such
		      that $\ell=\ell_1$ and $\sigma_d(\ell_i) = \sigma_d(\ell)$ for
		      all $i$.

	\end{enumerate}
\end{definition}

Next, we need to connect these laminations with polynomials. The basin of attraction of infinity, $B_\infty$, is the set of points having an unbounded forward orbit under the polynomial. The Julia set is the boundary of the basin of attraction of infinity. Assume for the moment that the Julia set is connected. The polynomial restricted to its basin of attraction of infinity is conjugate to $z^d$ restricted to the complement of the unit disk. We define $\Psi$ as the conjugating function. To clarify, we define $\Psi$ such that the following diagram commutes, $\Psi$ is a homeomorphism, and $\lim_{z\to\infty} \frac{\Psi(z)}{z} = 1$.
\begin{center}

	\includegraphics[width=2in]{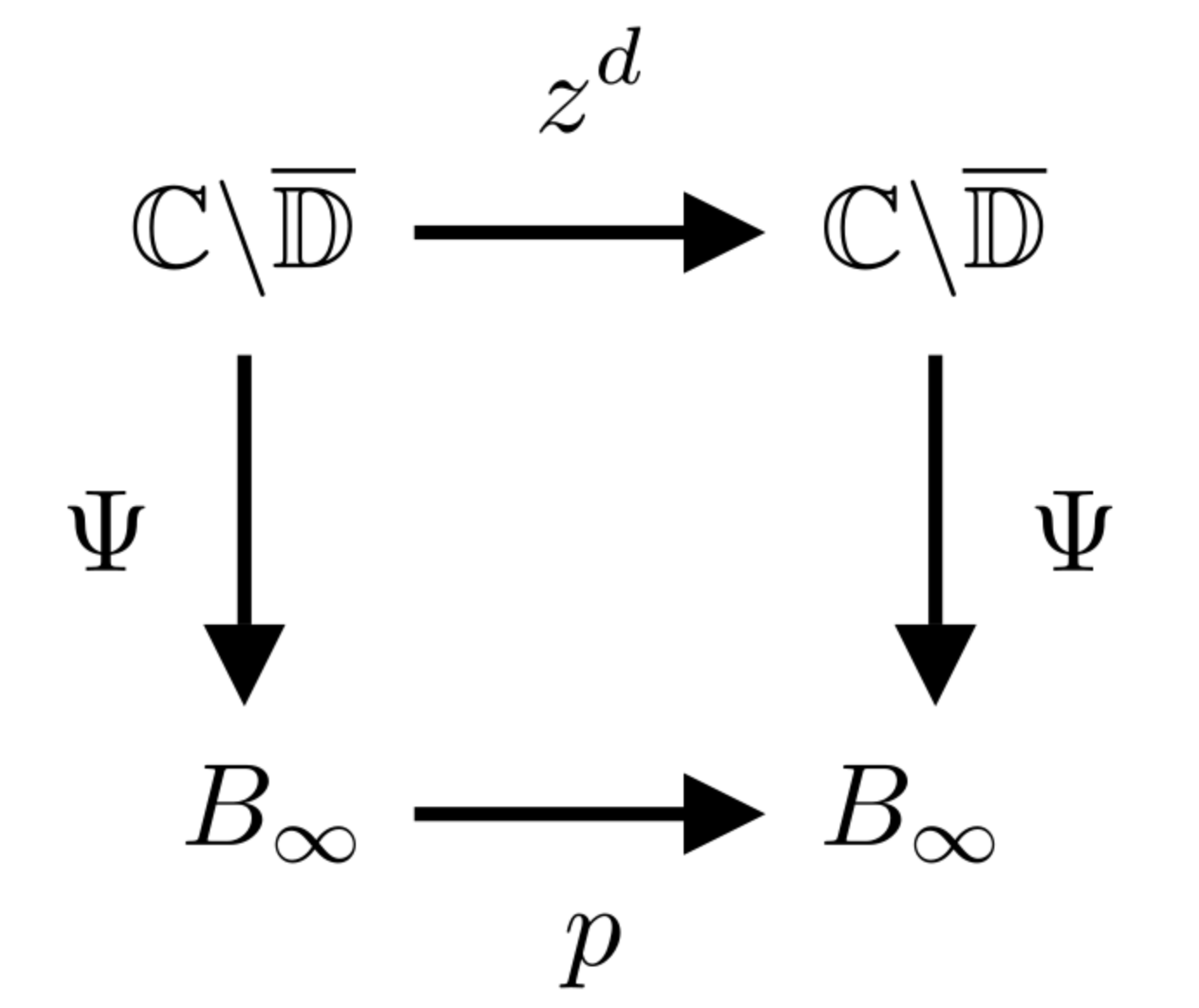}
\end{center}

The dynamical ray is the image of the straight ray under $\Psi$. The dynamic ray lands at a point in the Julia set. The lamination of a polynomial has the equivalence relation such that two angles are in the same class iff their dynamical rays land at the same point. This explanation is illustrated in a 5-minute video \cite{Hilton_Sirna_2023}. Dropping the assumption that the Julia set is connected and locally connected, a dynamical ray can be defined based on a potential function around the filled Julia set \cite{Orsay}, but it is no longer the case that all rays land, and it is harder to associate laminations with polynomials.

Laminations have been used to develop an exceedingly powerful and detailed model of the Mandelbrot set. However, it is much harder in higher degrees. In the case of cubics, there is uncertainty about the shape of the set of polynomials associated with the empty lamination. Indeed, the sheer dimension of the set of polynomials makes it difficult to understand the relationships between the polynomials of invariant laminations. Therefore, we introduce a set of simple laminations with a straightforward connection to parameter space, and we aim for a less detailed description of parameter space.

This document explores the uses of finite laminations for studying polynomials. \cref{CCP} introduces the terminology for discussing such laminations and proves which finite laminations are part of an invariant lamination. In \cref{CSP}, we count sibling portraits, which we introduced in \cref{CCP}. In \cref{FDL}, we introduce Finite Dynamical Lamination, or FDL, a definition that helps us extend combinatorial questions and helps us study the structure of parameter space. \cref{Limit} discusses the limits of sequences of FDL, and \cref{Realized} proves that all FDL are realized by a polynomial. \cref{Loops} finally shows how to assemble graphs that represent parts of parameter space.

\section{The Critical Chords Problem}\label{CCP}

Given a finite set of disjoint, periodic polygons, when does there exist a $\sigma_d$-invariant lamination that contains them? From \cite{thu09}, it is known that the question is the same as whether there is a set of $d-1$ disjoint critical leaves for that set of polygons. This section is dedicated to proving the existence of those critical cords when they exist.

\begin{definition}
	Given a gap, its \emph{basis} is its intersection with the circle, and its \emph{holes} are the circle arcs complementary to its basis.
\end{definition}

\begin{definition}
	A subset of the circle is \emph{critical} if it maps non-injectively by $\sigma_d$. A set in the disk (a gap or chord) is \emph{critical} if its basis is critical.
\end{definition}

A lamination may be called \emph{forward invariant}, meaning that the lamination contains its $\sigma_d$-image. The next theorem is well known. The following statement is proposition II.4.5 of \cite{thu09}.

\begin{theorem}\label{criticalcords}
	Given $d-1$ critical leaves compatible with a forward invariant lamination, there exists an invariant lamination containing the original that is also compatible with the critical chords.
\end{theorem}

\begin{definition}
	A \emph{round gap} is the closure of a component of the complement of a lamination that has arcs of the circle in its boundary. The empty lamination has one round gap.
\end{definition}

\begin{definition}
	A \emph{polygon} is the convex hull of a finite, non-singular equivalence class of the circle.
\end{definition}

A \emph{gap} is meant to be a component of the complement of the union of a lamination with the unit disk. The \emph{gaps} of a finite (one with finitely many leaves) lamination are defined as its polygons and round gaps. Lone leaves are polygons, and the term \emph{gap} always includes them.

\begin{definition}\label{degree}
	A gap, $G_i$, is of \emph{degree} $d_i$ if $\sigma_d|_{Bd(G)}$ is a positively oriented covering map of degree $d_i$. If the image of $G_i$ is a point then $d_i$ is the size of $G_i$'s basis. When the image is the leaf from $a$ to $b$, then having degree $d_i$ means that there are $2d_i$ vertices of the polygons and that the polygon's vertices alternate between the preimages of $a$ and $b$.
\end{definition}

Since we mostly think about polygons of a lamination, as opposed to its leaves, and those polygons have siblings, we give a name to a set of sibling polygons. Since a round gap with a degree is like the whole disk, this definition can be with reference to a round gap or the whole disk. We define sibling portrait in a way that is inconsistent with the existing notion of sibling collection: a sibling portrait \emph{of} a polygon $P$ is typically not composed of $P$'s siblings.

\begin{definition}\label{sibling}
	Given a round gap $R_i$ of degree $d_i$ and a polygon $P$ contained in $\sigma_d(R_i)$, a \emph{sibling portrait} is a set of disjoint polygons inside $R_i$ that map to $P$ and which together have $d_i$ sides mapping to each side of $P$.
\end{definition}

Since invariant laminations have siblings for every polygon, the question of whether we can create an invariant lamination comes down to whether we can find places for the siblings. That is, each polygon must be in some sibling portrait that could be added to the lamination. We work to show when this is possible. The next few proofs rely on the order of vertices in a polygon. This is the counterclockwise order of the vertices. When indexing the vertices of an $n$-gon, the subscript is implied to be modulo $n$.

\begin{lemma}\label{covering}
	A sibling portrait touches all the preimages of the original polygon's vertices, and all polygons of a sibling portrait have an equal number of vertices mapping to any particular vertex of the original polygon.
\end{lemma}
\begin{proof}
	Since polygons are disjoint even at their endpoints, each preimage vertex is, at most, part of two sides of one polygon. If the original polygon has an order of vertices like $x_1x_2x_3\dots x_n$ then each $x_k$ must be used the maximal number of times in preimages of $\overline{x_{k-1}x_k}$ and $\overline{x_kx_{k+1}}$. If we have a 2-gon, then we have an order $ab$ and in that case, each vertex must be used in two sides of a polygon that map to the two sides of the original polygon. Consider a walk around the sides of a polygon. We cannot afford to visit vertices in an order other than that of the original polygon. This excludes not only leaves whose images are interior leaves but also critical leaves. Moreover, we can only return to our starting place if we are at the end of a cycle.
\end{proof}

\begin{remark}
	The preimages of the vertices of a polygon are evenly spaced throughout the circle in the same repeating order as the polygon. (To see this, consider pulling back the polygon according to an all-critical d-gon; a half-open arc under any side of the d-gon maps monotone onto the circle.) For this reason, there is always at least one sibling portrait that can be formed by using all the vertices in one polygon.
	If a polygon maps with a degree (with positive orientation), then its siblings are in the same repeating order as itself. \end{remark}

This condition was first noted to be necessary in \cite{blokh2012laminationslanguageleaves}:

\begin{lemma}\label{iffpositive}
	Given a polygon, $P$, that maps as a covering map, it divides its siblings into even groups iff it maps with positive orientation. In this case, there exists a sibling portrait containing $P$ in the disc.
\end{lemma}
\begin{proof}
	Label the vertices of the original polygon, $\sigma_d(P)$, as $x_1x_2x_3\dots x_n$. Since the $\sigma_d(P)$ has a leaf from each $x_k$ to each $x_{k+1}$, $P$ must also. Therefore, as we walk around the leaves of $P$, we find vertices in a, possibly repeating, order of either $x_1x_2x_3\dots x_n$ or $x_nx_{n-1}x_{n-2} \dots x_1$. Since a polygon is a convex hull, the order of its vertices that we get by walking its leaves is the same as the order that they appear on the circle.

	Suppose that we start drawing $P$ one leaf at a time starting with $x_1$ and draw the leaves in CCW, counterclockwise, order. In all of $R_j$, we have the preimages of the vertices in order like $x_1x_2x_3\dots x_nx_1x_2x_3\dots x_n\dots$. Suppose that we then connect from $x_k$ to $x_{k-1}$ while moving CCW by making $P$ orientation reversing. Then the open arc under that leaf would contain $x_{k+1}$ and $x_{k-2}$, the start and end, respectively. Using the repeating order of vertices in the arc, we see that it contains fewer $x_k$ vertices than $x_{k+1}$. Thus, we cannot form complete groups of vertices as needed in the previous lemma.

	Alternately, if $P$ is orientation preserving, then under the side from $x_k$ to $x_{k+1}$, the first letter in the open arc is $x_{k+1}$ and the last letter in the open arc is $x_k$. Thus, under that side, there is some whole number of cycles.
\end{proof}

The following lemma is the contribution of Dr. Oversteegen. The proof does not require that $\overline{xy}$ is non-crossing with $\overline{ab}$, but it does require that they not be siblings.
\begin{lemma}\label{LexLemma}
	Given a leaf $\overline{ab}$ and a leaf $\overline{xy}$ such that the image $\sigma_d(\overline{xy})$ is disjoint from $\sigma_d(\overline{ab})$, either side of $\overline{xy}$ has the same number of siblings of $a$ as it does of $b$.
\end{lemma}
\begin{proof}
	Consider that the siblings of $a$, $b$, $x$, $y$ are evenly spaced throughout the circle in a fixed order. We label the siblings of these 4 points without distinguishing one sibling from another. If the $x$ or $y$ is equal to some sibling of $a$ or $b$, then their images obviously cannot be disjoint. Assume without loss of generality that $x$ is in some arc bounded by $a$ and then $b$. If $y$ is in an interval bounded by a sibling of first $b$ and then $a$, then the repeating order of the 4 points is $axby \dots$. A half open arc starting and ending with $a$ will map monotone onto the circle. Thus, the image of $\sigma_d(\overline{xy})$ crosses $\sigma_d(\overline{ab})$.
\end{proof}

\begin{lemma}\label{addsiblings}
	Given a disjoint, finite, forward invariant set of orientation preserving polygons, there exists a lamination containing those polygons such that all polygons are part of a sibling portrait in the disk.

	Or, more generally, given a disjoint, finite set of orientation preserving polygons such that for every pair of polygons, $P_1$ and $P_2$, either $\sigma_d(P_1)=\sigma_d(P_2)$ or $\sigma_d(P_1)$ is disjoint from $\sigma_d(P_2)$; then there exists a lamination containing those polygons such that all polygons are part of a sibling portrait in the disk.
\end{lemma}

\begin{proof}
	As long as there is a polygon, $P$, left that is not part of a sibling portrait, we can add a polygon that is its sibling. Consider a sibling $a'$ of a vertex of $P$. $a'$ is on the boundary of a round gap, call it $R$. Every sibling of a vertex $P$ is either in an open arc of $R$ or in the closure of one of its holes. These holes fall into one of 2 categories: either they are bounded by a leaf that is part of a polygon sibling to $P$, or they are not. In the first case \cref{iffpositive} applies to that polygon; thus, on the side of the polygon containing $R$, there is a complete set of vertices. Using \cref{LexLemma}, we can make the same conclusion about the other type of hole. Each closure of a hole contains an equal number of siblings of each vertex/letter. Since each vertex has the same number of siblings in the whole disk, each vertex has the same number of siblings in the union of the open arcs of $R$.

	The vertices in the open arcs of $R$ are in the same positive order as $P$. Supposing the opposite, then some vertices are skipped. But since they are in positive order in the disk, some hole must contain the skipped vertices. In order to skip a vertex, you have to separate it without putting the other vertices with it. Thus, some closure of a hole of $R$ contains an uneven number of vertices in contradiction to the argument above.

	Since those vertices are together in one gap and positively oriented, we can form a new polygon that is also positively oriented.
\end{proof}

\begin{lemma}\label{freecriticality}
	\mbox{}
	\begin{enumerate}[i)]
		\item
		      If $R_j$ is a round gap with of degree $d_j$, and $P$ is a polygon contained in $\sigma_d(R_j)$, then the round gaps resulting from inserting a sibling portrait into $R_j$ each have a degree.
		\item Consider the set of gaps (including polygons) $\{G_i\}$ formed when inserting that sibling portrait into $R_j$, and their degrees ${d_i}$. Then $$\sum_i(d_i-1) = d_j-1$$.
		\item Moreover, given a lamination that consists of finitely many disjoint polygons such that each polygon is part of a sibling portrait in the whole disk, all gaps have a degree, and the equation above holds. In this case $d=d_j$.
	\end{enumerate}
\end{lemma}
\begin{proof}
	\mbox{}
	\begin{enumerate}[i)]
		\item
		      The new polygons are already guaranteed to have a degree by \cref{iffpositive}, so consider only the round gaps created by inserting a sibling portrait into $R_j$. Consider an original polygon with some order, say, $x_1x_2x_3\dots x_n$. In all of $R_j$, that same order repeats. By the previous lemma, up to starting point and number of cycles, all polygons of the sibling portrait will have an order like $x_1x_2x_3\dots x_n$. While we construct a sibling portrait, we add one polygon at a time. Each polygon divides what remains of the circle into arcs. If the arc is of degree one, then of course, we have everything we could want. If we have a critical arc, then there are some preimages of the original vertices littered in the arc.

		      By \cref{iffpositive} after some ``first'' polygon is inserted, a critical arc created must have, up to starting point, the same repeating order as $R_j$. Let the start of the arc be $x_k$. Thus, its end is $x_{k+1}$. Moreover, the first preimage of a vertex appearing in the open arc is also $x_{k+1}$. We can also note that the sub-arc at the beginning and ending of the critical arc are siblings as well as being consecutive in the circular order of the arc. We then add a polygon starting with $x_{k+1}$ since that is the first preimage vertex of the open arc, and thus the last point of the polygon is $x_k$ by the polygon's circular order. In this way, as we remove from the arc by adding polygons until there are no critical arcs, what remains is a number (at least 2) of sibling arcs of the circle. These arcs and leaves of the first and subsequent polygons form the boundary of a round gap. From the parent polygon, we take the leaf $\overline{x_kx_{k+1}}$, and from every child polygon, we also take the leaf $\overline{x_kx_{k+1}}$. Since the arcs in between the leaves of this gap all have the same endpoints, they are also siblings. Therefore, the round gap that is formed has some degree equal to the number of polygons in its boundary.

		\item
		      Fix some sibling portrait. Once again, choose a ``first'' polygon and focus on one side. The degree of the round gap immediately under the first polygon is one greater than the number of polygons immediately under that side of the first polygon. The degree minus 1 of a polygon is the number of extra sets of vertices that it touches. Under each polygon below the first polygon, the subsequent round gaps have degrees that are one greater than the number of polygons immediately below them. According to this counting, each polygon adds degree to the round gap above them. The total $\sum_i(d_i-1)$ of all the gaps under one side of the first polygon is the number of sets of vertices under that side of the polygon. Since the first polygon takes up one set of vertices besides its own excess degree, the total over all sides and its own excess degree is one less because of that polygon.

		\item
		      Consider working up to the lamination inductively, but one sibling portrait at a time. The statement is true of the empty lamination. At each step while inserting a sibling portrait, in-between other sibling portraits, we can see it as inserting a sibling portrait into a round gap with, say, degree $d_j$. Since a round gap with a degree is like the disk, the equation holds for the new gaps inside the round gap. Since the new terms of the sum are equal to what they replace, the equation holds.
	\end{enumerate}
\end{proof}

The formula above implies something comparable to the Riemann-Hurwitz relation found in \cite{alma994567553503176} when it is applied to the polygons and the round gaps of the lamination. That is the criticality trapped in the polygons plus the criticality free in the round gaps adds to $d-1$.

After forming the sibling portraits, we can place $d-1$ critical chords (without forming a polygon out of them) into gaps in at least one way. Any point in the basis of a critical gap could be the starting point of such a concatenation of critical chords. These critical chords separate the circle into $d$ branches of the inverse, which by \cref{criticalcords} creates an invariant lamination containing the original set of polygons.

\begin{theorem}\label{existsinvariant}
	Given a forward invariant set of disjoint polygons, they are in at least one invariant lamination iff all the polygons map forward as positively oriented covering maps.
\end{theorem}

It is further asked whether the critical chords can be chosen to touch the original polygons. Consider sliding a critical chord to one side or the other until it can't slide any further without hitting one of the original polygons. Eventually it will stop as long as the lamination is non-empty.

Another question is whether there exists an invariant lamination that contains every periodic orbit in some list and no more periodic polygons. This is sometimes not the case. For one example, in degree 3, take $L$ to be the rotational triangles of \_001 and \_112. The leaf from \_0 to \_1 is forced. In the original lamination, the round gap where the leaf is forced is not critical. For an example in degree 2, take $L$ to be the orbit of the leaf from \_000101 to \_101000. A rotational triangle is forced in the non-critical part of the critical gap.

\begin{conj}
	Given a disjoint set of positively oriented periodic polygons, there exists an invariant lamination such that all the additional periodic polygons are of lower period than the originals.
\end{conj}

Now consider generalizing \cref{existsinvariant} by (temporarily) expanding the definitions involved. First recall that a lamination is defined to include all points of the circle. A forward invariant lamination could contain a wandering triangle or a critical leaf with a pre-periodic endpoint or any number of other things. A finite forward invariant set has finitely many leaves (degenerate leaves don't count).

\begin{theorem}
	Given a finite forward invariant lamination, it is a subset of at least one invariant lamination iff each polygon in the forward invariant set either has a point as its image or maps as a positively oriented covering map.
\end{theorem}

\begin{proof}
	Take the interpretation that a point of the circle counts as a polygon. Of course, it has no sides, so when we talk about ``all sides'', we make a vacuously true statement. Without changing the definition of sibling portrait (\cref{sibling}), the sibling portrait of a point can be a set of points. It could also include a critical leaf. Thus, a critical leaf is always part of a sibling portrait in a lamination. The definition of degree (\cref{degree}) plays nicely with this rereading of the term polygon. All the lemmas of the paper hold just as well with, at worst, an added case in the proof. \cref{criticalcords} requires the stipulation that a critical chord is compatible with a lamination that contains it. This interpretation is traditional under the convention that \emph{compatible} means the union is a lamination.

\end{proof}
\begin{conj}
	The word finite in the theorem above is unnecessary.
\end{conj}

\section{Counting Sibling Portraits}\label{CSP}

Given a round gap $R$ of degree $i$ and an $n$-gon, $P$, contained in $\sigma_d(R)$, let $F(i,n)$ denote the number of sibling portraits, and let $f(i,n)$ denote the number of sibling portraits with all polygons one-to-one. We observe a bijection with the full rooted $n$-ary trees, which are counted in \cite{rusu2021raney}. Though it is a function in two variables, that does not stop us from providing OEIS, \cite{OEIS}, references: $f(i,2)$ = A000108$(i)$, $f(i,3)$ = A001764$(i)$, $f(i,4)$ = A002293$(i)$, $f(i,5)$ = A002294$(i)$. These are the Fuss–Catalan numbers.

\begin{theorem}\label{firstcount}
	$$ f(i,n) = \frac{{ni \choose i}}{(n-1)i+1}$$
\end{theorem}
\begin{proof}
	One-to-one sibling portraits of an $n$ sided polygon in a degree $i$ round gap are in bijection with full, rooted, $n$-ary trees with $i$ internal vertices. Polygons correspond to internal vertices, and the trees are made according to the principle that the first CCW new polygon divides the region that it is in into $n$ smaller regions. In the bijection pictured in \cref{countfig1}, a depth-first traversal of the internal vertices in the tree would correspond to a CCW ordering of the polygons based on their first ray. Such trees are known to be counted by the closed-form formula above \cite{rusu2021raney}.

	\begin{figure}[ht]
		\includegraphics[width=6in]{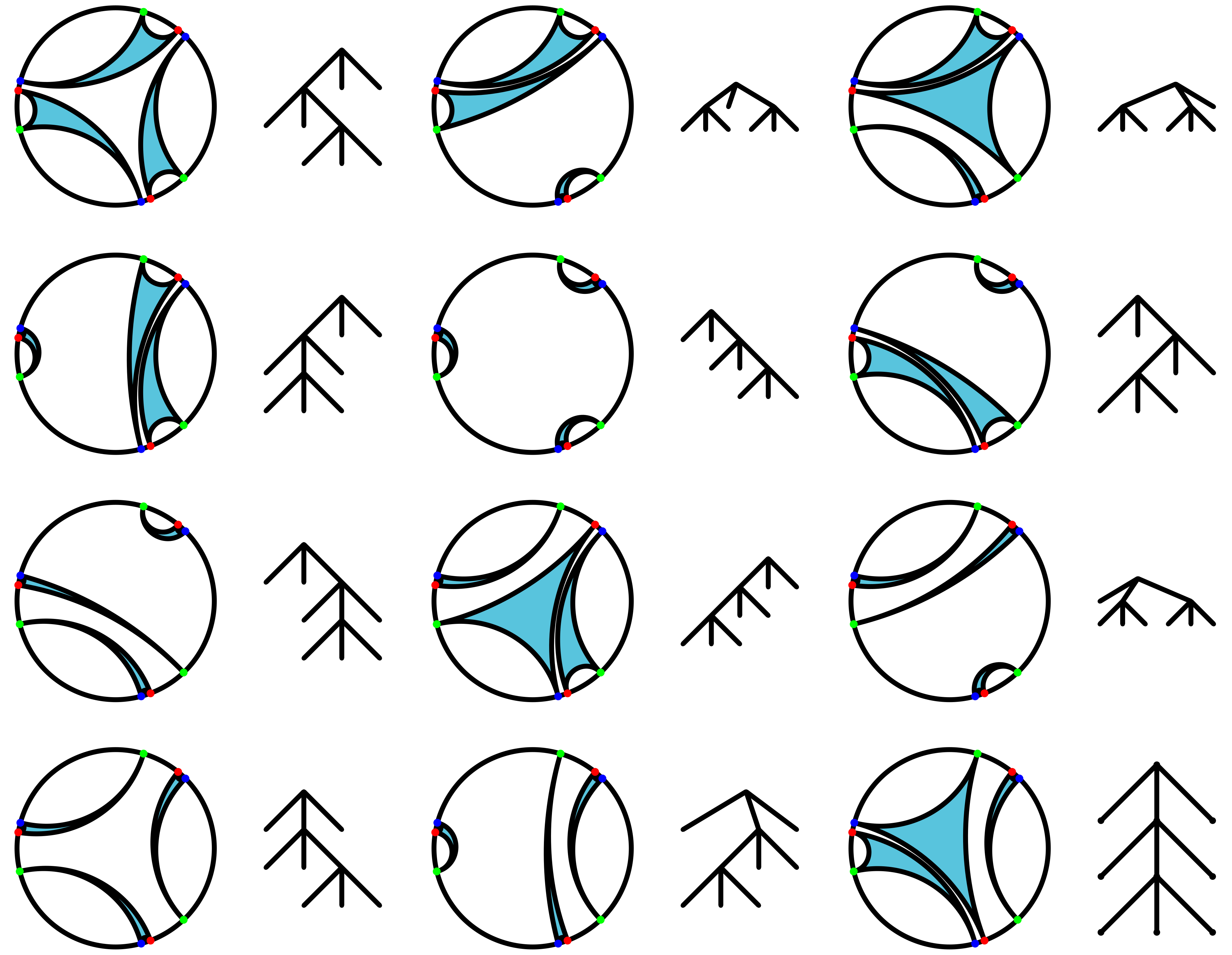} \\
		\caption{The bijection considered in the proof of \cref{firstcount} in the case of a degree three round gap with a triangle in its image.}
		\label{countfig1}

	\end{figure}
\end{proof}

\begin{theorem}\label{count2}
	$$F(i,n) = f(i,n+1)$$
\end{theorem}
\begin{proof}
	Let $B$ be the following function from the one-to-one portraits for $i$ and $n+1$ to the complete set of sibling portraits for $i$ and $n$. Take a sibling portrait from the domain. It is essentially a way of connecting $n+1$ colors, each with $i$ points into disjoint polygons. Selecting a color arbitrarily, $X$. Collect the polygons into groups by transitively adding the polygon that appears after the $X$ vertex of any polygon. The output sibling portrait is a set of convex hulls of the sets of vertices other than $X$ from all the polygons in each of the groups.

	\begin{figure}[ht]
		\includegraphics[width=6in]{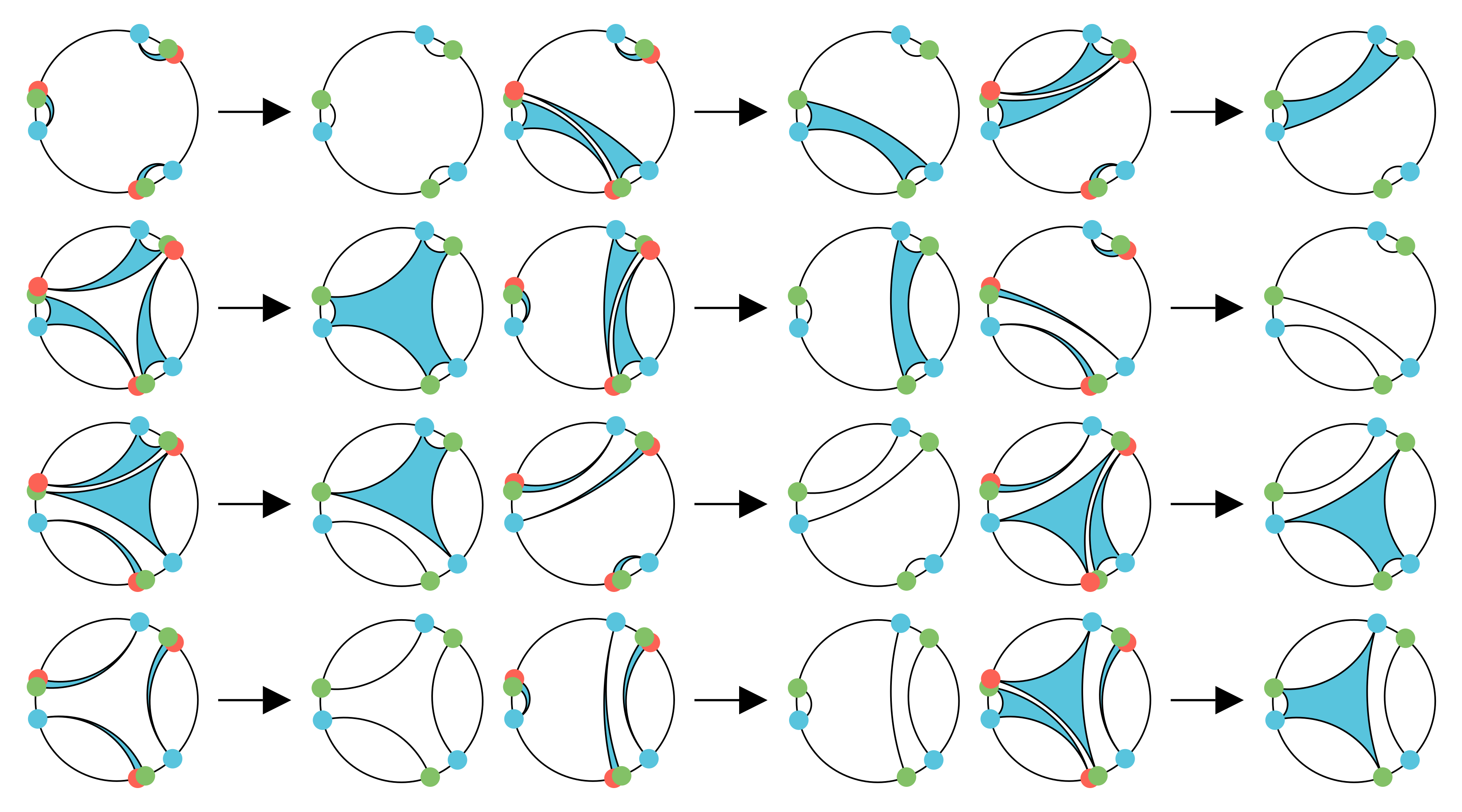} \\
		\caption{The bijection considered in the proof of \cref{count2} where the domain is the case of a degree three round gap with a triangle in its image.}
	\end{figure}

	$B$ has the ascribed range because at neither step could it cause polygons to cross. It first merges adjacent polygons. Next, it removes vertices, which is replacing a polygon with a smaller polygon (the convex hull of a subset of the vertices); therefore, it will still be a lamination.

	Without loss of generality, assume that the domain is the sibling portraits of a polygon with vertices in the positional order $abcd$. Let the range be sibling portraits of polygons of $abc$, and we correspond the vertices of each in the way suggested by the labeling, and of course $X=d$.

	We form the inverse of $B$. By \cref{covering,iffpositive}, all the polygons of a sibling portrait map as a positively oriented covering map. Consider a polygon from an element of the range. It is in the form $abcabcabc\dots$ in positional order. In the whole circle there are multiple $d$ points that we could add, but we are forced to add the ones appearing right before the $a$ vertices of this polygon. (Expanding the polygon in that way cannot cause things to cross since it is the $d$ directly before the $a$). Next we need to break this into polygons, and the only way to do this consecutively is ending with $d$. Suppose by way of contradiction that there is some $a$ vertex that we are allowed to connect a $d$ vertex to other than the $a$ preceding it, and finish that polygon as $P_1$. Consider an arc of the circle partitioned by that polygon and the vertices in that arc. At least one arc will have a sequence of vertices that does not start with $a$, and the set of polygons formed in that arc will not be placed with $P_1$ by $B$. Thus, there is exactly one sibling portrait of the domain for each of the range.
\end{proof}

\section{Finite Dynamic Laminations}\label{FDL}

\begin{definition}

	A lamination $\mathcal L$ is called a \emph{Finite Dynamical Lamination} or an \emph{FDL}, in degree $d$ if it is non-empty and satisfies these properties:
	\begin{enumerate}
		\item $\lam$ has finitely many leaves;
		\item for each leaf $\overline{ab}\in\mathcal{L}$, $\sigma_d(a) \neq \sigma_d(b)$;

		\item for each leaf $\ell\in\mathcal{L}$, $\sigma_d(\ell)\in\mathcal{L}$ (no critical leaves);

		\item there is a whole number, $n \geq 0$, such that each non-periodic leaf $\ell$ has a preimage in $\lam$ iff $\sigma^{n-1}_d(\ell)$ is periodic, and each periodic leaf has a non-periodic preimage in $\lam$ iff $n>0$;
		\item for each non-periodic leaf $\ell\in\mathcal{L}$, there exist $\mathbf d$ {\bf
				      disjoint} leaves $\ell_1, \dots, \ell_d$ in $\mathcal{L}$ such
		      that $\ell=\ell_1$ and $\sigma_d(\ell_i) = \sigma_d(\ell)$ for
		      all $i$;
		\item for each leaf $\ell$, $\ell\in\mathcal{L}$ iff $\ell$ is on the boundary of a convex hull of an equivalence class of $\lam$;
		\item the leaves of any periodic equivalence class of $\lam$ map as a covering with positive orientation.
	\end{enumerate}
\end{definition}

Given that the endpoints of all the leaves of such a lamination are eventually periodic, those endpoints are rational angles. Regarding the following definition, \emph{land} is taken to have the same meaning as in \cite{Orsay}, which establishes that a rational ray will always either land or crash into a critical point. We will show that the set defined is a non-empty set (\cref{realized}).

\begin{definition}
	The \emph{co-landing locus} of an FDL is the set of polynomials under which two dynamic rays land at the same point if their angles share an equivalence class of the FDL.
\end{definition}

\begin{proposition}
	Consider a polynomial with a lamination, $\mathcal L$. The polynomial is in the co-landing locus of an FDL, $\mathcal L'$, iff the equivalence relation of $\mathcal L$ is finer than the equivalence relation of $\mathcal L'$.
\end{proposition}

The following sections first introduce the pullback tree, and then later the generational FDL graph. The latter is perhaps the more illuminating structure for understanding parameter space, but its vertices are those from a single level of the former.

It is possible to surjectively parameterize the set of FDL, with a natural number, $n$, a set of periodic leaves, $P$, and a pullback scheme. We take a slice of the set of FDL by specifying $P$.

\begin{definition}
	Given a set of periodic leaves constituting an FDL, $P$, the \emph{pullback tree} is defined as the graph where
	\begin{enumerate}
		\item All FDL having exactly $P$ as its periodic leaves are the vertices.
		\item Two vertices (laminations) are joined by an edge iff one is the image of the other, and they are different.
	\end{enumerate}
\end{definition}

The pullback tree is connected since all added leaves are preperiodic. The pullback tree is a tree, since any loop would require some lamination to have two images or that some lamination eventually, but not immediately, maps to itself. The latter case is impossible because each lamination must include its image as a subset. Therefore, we additionally impose that:
\begin{enumerate}
	\setcounter{enumi}{2}
	\item The root of the pullback tree is the lamination that is its own image.
\end{enumerate}

With this stipulation, we establish the meaning of the terms parent and child. We see that all FDL have children from \cref{existsinvariant}.

\begin{figure}[ht]
	\includegraphics[width=\linewidth]{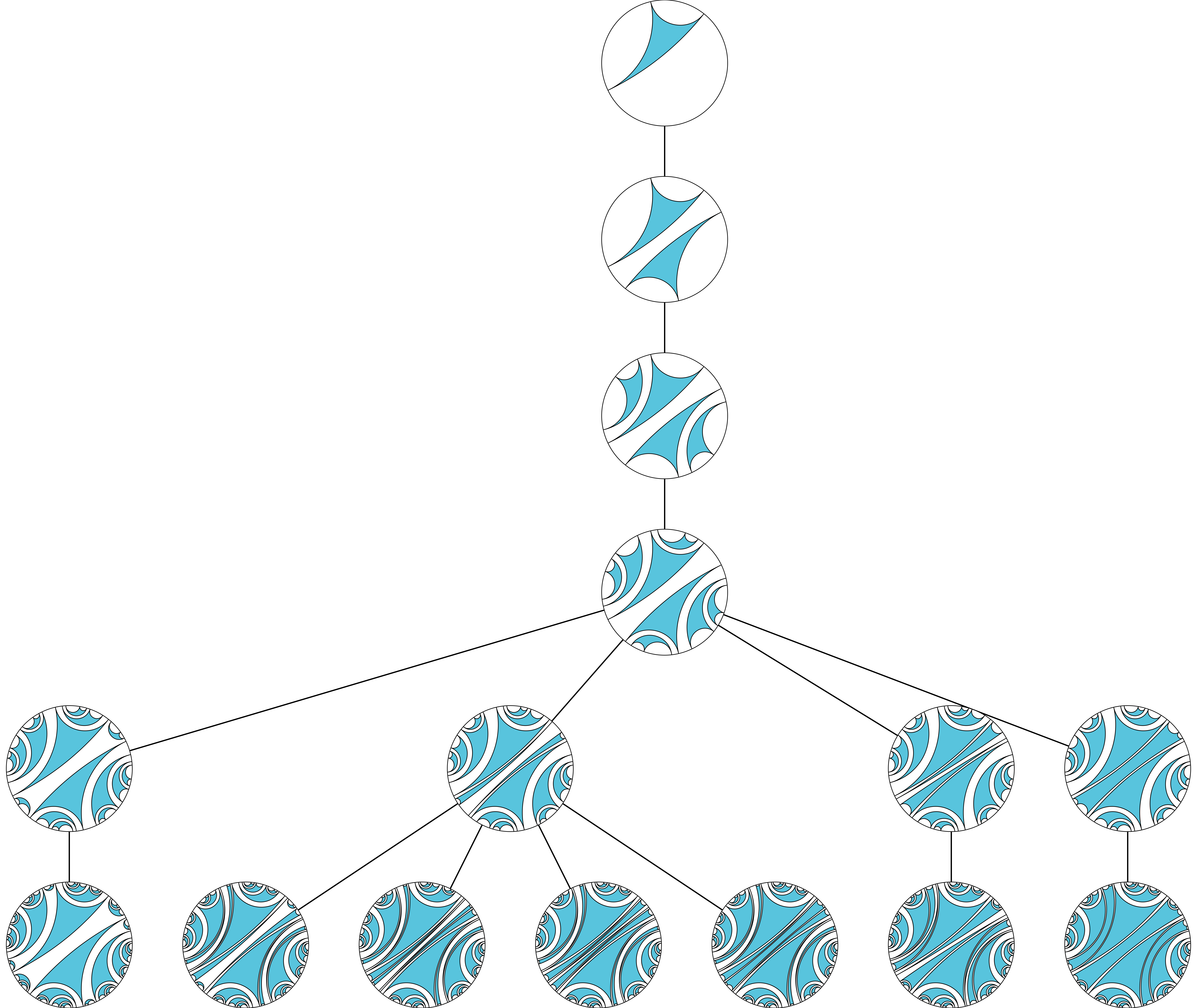}
	\caption{The pullback tree starting with the polygon with vertices \{$\_001$,$\_010$,$\_100$\} up to level 5, also known as the rabbit pullback tree.}
\end{figure}

\section{What are the limit sets of FDL?}\label{Limit}

The goal of this section is to establish some facts about the limit of a sequence of FDL. Nothing good follows from the definition of FDL. The set of finite dynamical laminations is not closed and contains none of its limits. The term ``limit lamination'' has an existing meaning\cite{blokh2015combinatorialmandelbrotsetquotient}, so the limit of a sequence of FDL will be referred to as an ``end lamination.'' First, we situate FDL inside a larger set of laminations. Then after considering how bad sequences of FDL are in general, we impose a fairly obvious restriction and find that such sequences limit to reasonably nice laminations.

Of course, we wish to classify these sequences by how nice their limits are. For that purpose, we summarize the definitions of a few terms that we use sparsely. The most important term is \emph{q-lamination}, which is defined in \cite{blokh2012laminationslanguageleaves} as a lamination having an equivalent relation that is \emph{laminational} with the added convention that a leaf is in the lamination iff it is in the boundary of the convex hull of an equivalence class. \emph{Laminational} is also defined in \cite{blokh2012laminationslanguageleaves} and refers to an equivalence relation that meets several conditions that are each necessary for the lamination to be realized by a polynomial in the sense contemplated here. One such condition is that the equivalence classes are finite. These terms generally presuppose that the lamination is invariant. Unclean is defined in Thurston as having 3 leaves meet at a point. Unclean laminations are not q-laminations, but they might be \emph{proper}, in which case their equivalence relations are laminational. When we say that a lamination is not proper, we mean that it has infinite equivalence classes. To summarize, an unclean sibling lamination can be cleaned into a q-lamination iff it is proper.

In our examples, which are quadratic, we give a few more specific labels, mainly from \cite{blokh2015combinatorialmandelbrotsetquotient}. We say that a lamination is hyperbolic if there is a critical gap with uncountably many sides. If there is an uncountable gap that returns with degree 1, then we say that it and the lamination are \emph{Siegel}. We call a gap caterpillar if it has countably many leaves. In accordance with folklore, we say that a lamination is sub-hyperbolic if there is a finite equivalence class that is critical and preperiodic.

\begin{definition}
	The distance between two leaves is the sum of distances on the circle between the first endpoints of the leaves and the second endpoints of the leaves (we decide which endpoint is first in such a way that it minimizes the distance). The distance between two laminations, $\mathcal A$ and $\mathcal B$, is the Hausdorff distance from the set of leaves and degenerate leaves of $\mathcal A$ to the set of leaves and degenerate leaves of $\mathcal B$ taking the distance leaf-wise.
\end{definition}

The inclusion of the degenerate leaves in the distance computation is harmless, but seemingly optional, and is not used to any particular advantage. The perceived benefit of this distance metric is that the distance from one FDL to its child is at most the length of the longest leaf in their complement. In fact, we can conjugate the round gap that we are placing the sibling portrait inside with a circle, and up to a constant multiplier, the distance from the original lamination to the one containing the sibling portrait will be the same as from the circle to the conjugated copy of the sibling portrait.

\begin{lemma}\label{compact metric space}
	The set of laminations is a compact metric space.
\end{lemma}
\begin{proof}
	The maximum distance from one lamination to another is at most the greatest distance from one leaf to another, which is $\frac{1}{2}$. As for the closed part of the claim, suppose by contradiction that the limit of a sequence of laminations had a pair of crossing leaves. The arcs subtended by those leaves overlap in an arc of length $\epsilon>0$. Select from the sequence, a lamination such that every leaf of the limit has a leaf of the lamination within $\frac{\epsilon}{3}$. The leaves of the lamination close to the crossing leaves in the limit would have to cross because the arcs subtending them must have an intersection of length at least $\frac{\epsilon}{3}$.
\end{proof}

\noindent
\begin{minipage}[t]{0.55\textwidth}
	\noindent
	The traditional approach to laminations, described in \cite{aziz2023fixedflowers}, is to take a set of critical chords and a rotational polygon and then pull back the polygons according to the critical chords infinitely many times and consider the limit, which is called a pullback lamination. Notably, the critical chord is often not in the end lamination, but is often pictured in blue with the lamination (Pictures showing this were made using \cite{Falcione}.) The canonical lamination is the lamination generated using a canonical critical portrait in the sense of \cite{aziz2023fixedflowers}. We diverge from this approach by trying to focus on the FDL in the sequence instead of the critical chords. In the Appendix, we go so far as to develop the notion of canonical pullback laminations without critical chords.
\end{minipage}%
\hfill
\begin{minipage}[t]{0.4\textwidth}
	\vspace*{0cm}
	\includegraphics[width=0.8\linewidth]{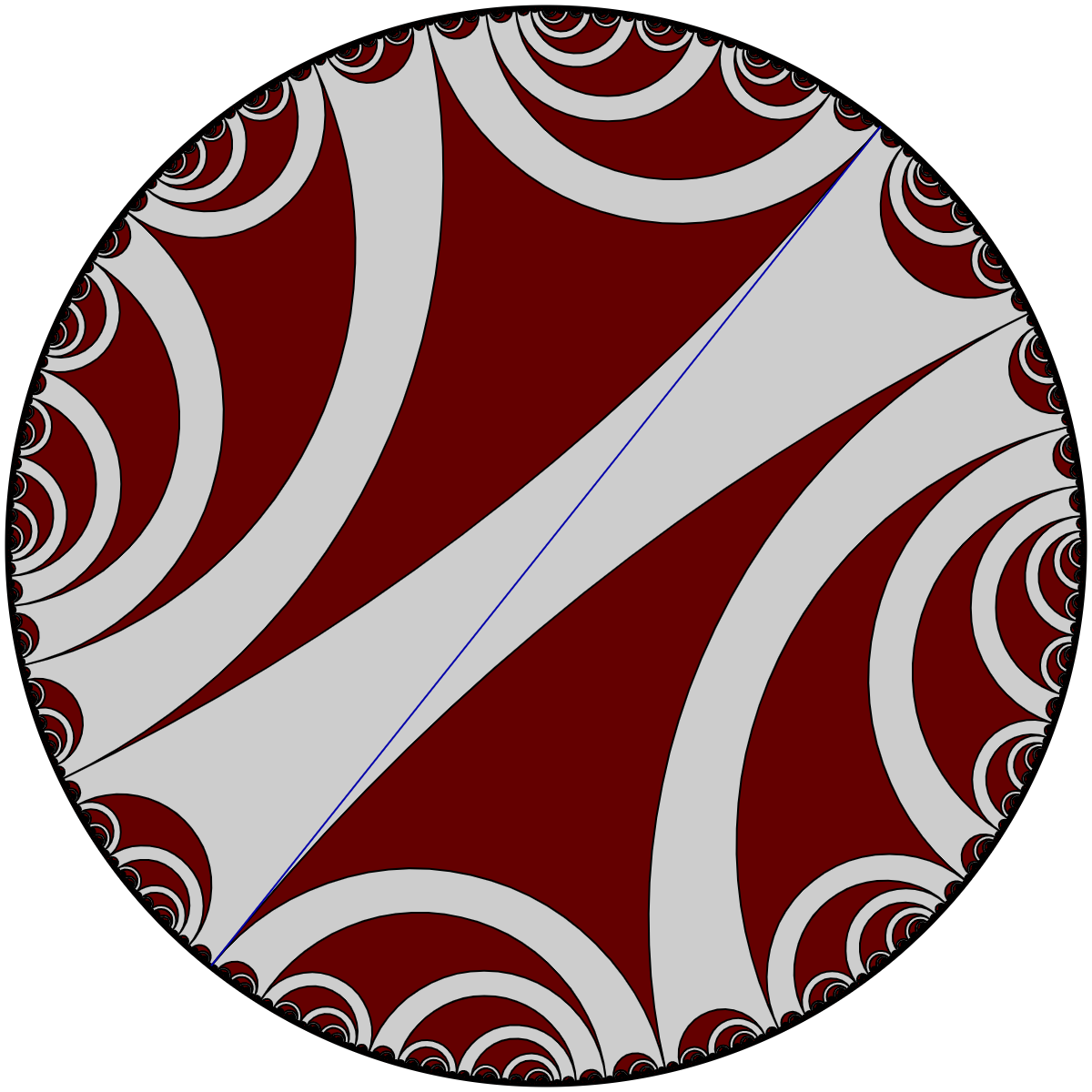}
	\captionof{figure}{The canonical rabbit pullback lamination.}
	\label{fig:canonical}
\end{minipage}

Of course, a sequence of FDL might be a constant sequence or contain a constant subsequence. Suppose it does not. It might oscillate. Suppose it converges to a lamination not in the sequence. Surely then it will be an invariant lamination?

\begin{ex} \label{non-invariant}
	There is a non-constant converging sequence of FDL whose limit is not invariant. Consider the canonical rabbit lamination, and consider the triangles on the boundary of the image of the critical gap. Form a converging sequence of those triangles. Preferably the limit of this sequence should be the singleton at $\_010$.
	Next, for each triangle in the sequence, form the minimal FDL containing the hexagon that maps 2 to 1 onto that triangle. The limit of this sequence of FDL is the rabbit lamination with a critical leaf touching the periodic point. But that critical leaf has no preimage.
\end{ex}

\begin{definition}
	A sequence to the \emph{end of the pullback tree} is an infinite sequence of FDL such that each lamination is the image of the next lamination and such that the first element of the sequence is the root of the tree. The limit of such a sequence is an \emph{end of the pullback tree}, and the set of \emph{ends of the pullback tree} is the \emph{base of the pullback tree}.
\end{definition}

Before we discuss the theory of such sequences, we should contemplate an intermediate kind of sequence of FDL. By that, we mean an element of the closure of the base of the pullback tree. The next example shows this distinction to be meaningful, but how is it a sequence of FDL? By the definition of convergence, such a lamination is the limit of a sequence of FDL. The stipulation on such a converging sequence of FDL is that $\forall \epsilon, \exists N, n>N \Rightarrow$ the distance from the FDL to the nearest element of the base of the pullback tree is less than $\epsilon$. \cref{non-invariant} is not such a sequence, and it could not be. We know from \cref{siblinginvariant} and the statement of compactness in \cite{blokh2015combinatorialmandelbrotsetquotient} that the closure of the base of the pullback tree contains only sibling invariant laminations. But there is nothing else nice to say about those laminations. We know that there are improper laminations in the closure of the base of the pullback tree. \par
\noindent
\begin{minipage}[t]{0.55\textwidth}

	\noindent
	\begin{ex}\label{ex56}
		The base of the pullback tree is not closed. Consider the sequence of invariant laminations such that each one is the invariant lamination containing each FDL from \cref{non-invariant}. The limit of the sequence is pictured and is not proper with a critical leaf touching the periodic point. To see that this is not in the base of the pullback tree, consider that it violates both \cref{proper} and \cref{criticalsides}.

		The base of the rabbit pullback tree contains a lamination with each hexagon or critical leaf whose image is under the minor of the canonical rabbit lamination. (The minor is the image of the longest leaf.)
	\end{ex}
\end{minipage}%
\hfill
\begin{minipage}[t]{0.4\textwidth}

	\vspace*{0cm}
	\includegraphics[width=0.8\linewidth]{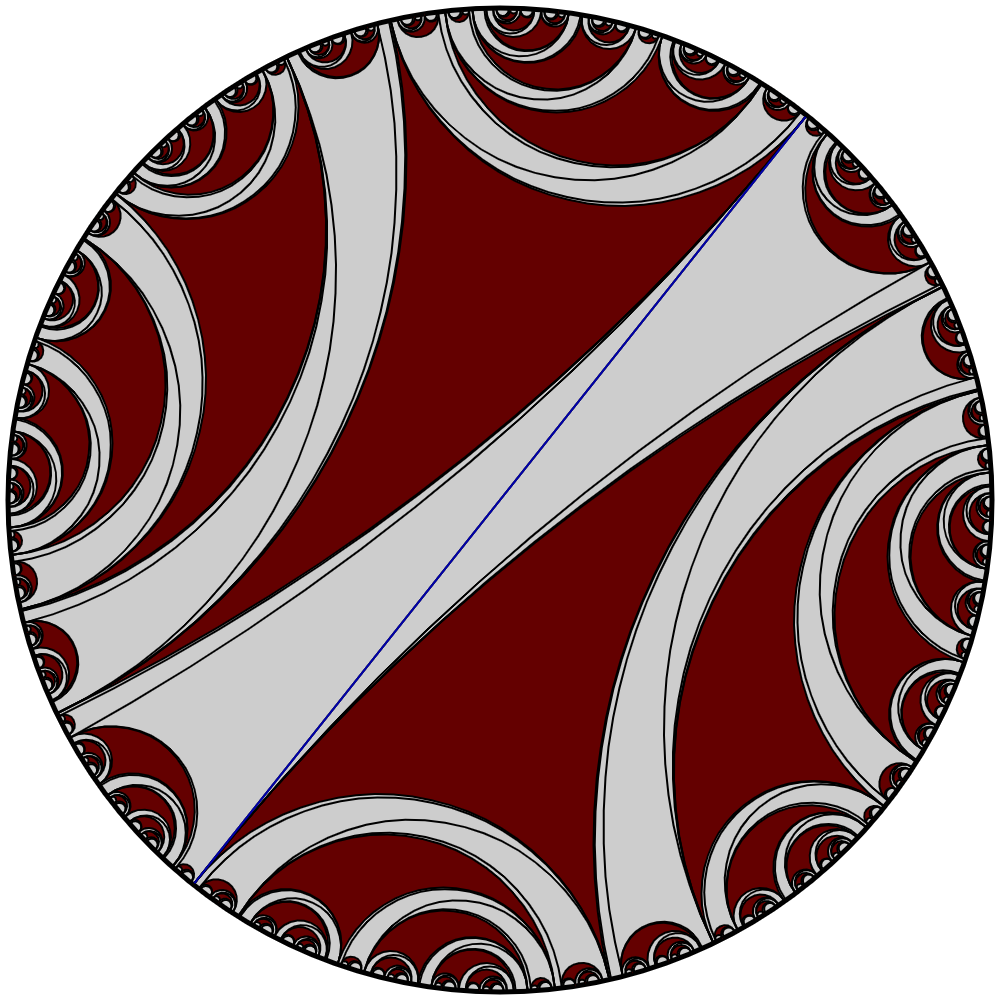}
	\captionof{figure}{\cref{ex56}}
\end{minipage}

\vspace*{0cm}
This shows that there is no lesser restriction on the sequences of FDL that provides reasonably useful laminations in the limit. Thus, we will spend considerable time on sequences to the end of the pullback tree because there is no greater generality to consider.

\begin{lemma}\label{limit contains}
	Any sequence to the end of the pullback tree converges, and the limit is a proper superset of any FDL in the sequence.
\end{lemma}
\begin{proof}
	By forward invariance, each lamination is a superset of the next. Indeed, the FDL grow with depth. Consider as a candidate limit, the closure of the union of all the FDL in the sequence. Since the FDL grow to eventually contain every element of the union which in turn has a leaf within epsilon of any leaf of the candidate lamination, the candidate is a limit of the sequence. By \cref{compact metric space}, the limit is unique.
\end{proof}

\begin{lemma}\label{siblinginvariant}
	The laminations in base of the pullback tree are sibling-invariant laminations.
\end{lemma}

\begin{proof}
	The claims in the definition of sibling invariant obviously apply to any leaf appearing in an FDL and the sequence limiting to this lamination, and by the definition of lamination those claims also apply to all degenerate leaves. Thus consider a limit leaf, and consider our leaf to have length $x>0$. Note that it must be the limit of leaves that are in some of those FDL otherwise it would not appear in the limit. Since $\sigma_d$ is a continuous function, the images of the leaves in the sequence converge to the image of the limit.

	Let $0< \epsilon < \frac{x}{2d}$. Consider the arcs of length $2\epsilon$ around each endpoint of the limit leaf. The sets of preimage arcs of those two arcs are disjoint. Consider a leaf from one of the FDL in the sequence that is within $\epsilon$ of the limit leaf. That leaf has $d$ disjoint leaves mapping to it and all of them are from one arc to another. The number of options for how to connect the arcs is essentially a sibling portrait, thus there are finitely many ways to connect them. Even if further leaves in the sequence have sibling collections that join the arcs differently, by the sheer cardinality of the sequence of sibling collection, they must accumulate somewhere, and the leaves of their limit will also join points from the disjoint arcs, and thus the limit sibling collection will remain disjoint.

\end{proof}

\noindent
\begin{minipage}[t]{0.6\textwidth}
	\begin{ex}
		There exists a lamination in the base that contains a critical leaf. Consider the rabbit and chose an endpoint for the limit critical leaf that is not a preimage of the original periodic vertices and is also not in the boundary of the critical Fatou gap of the canonical rabbit. Moreover, you need to ensure that the chosen critical leaf does not pass though the original periodic triangle but does pass through infinitely many preperiodic leaves of the canonical rabbit, though these may not be exactly sufficient conditions. Using the long-established technique of pulling back relative to a critical chord, this arrangement forces there to be a leaf at the endpoint of the critical chord.

		\cref{critical_leaf-001.png} has the critical leaf with angle $\frac{1}{8}$. It is an example of a sub-hyperbolic lamination where the criticality is eventually periodic of a different period than the periodic leaves in the root of the pullback tree. It seems that in this way we could create a critical leaf at a wandering point.

		What we can not do with this technique is force a lamination to have an additional leaf at a point that is in the boundary of the original critical Fatou gap. Thus, it seems that we can not cause a satellite bifurcation.
	\end{ex}
\end{minipage}%
\hfill
\begin{minipage}[t]{0.38\textwidth}
	\vspace*{0cm}
	\includegraphics[width=\linewidth]{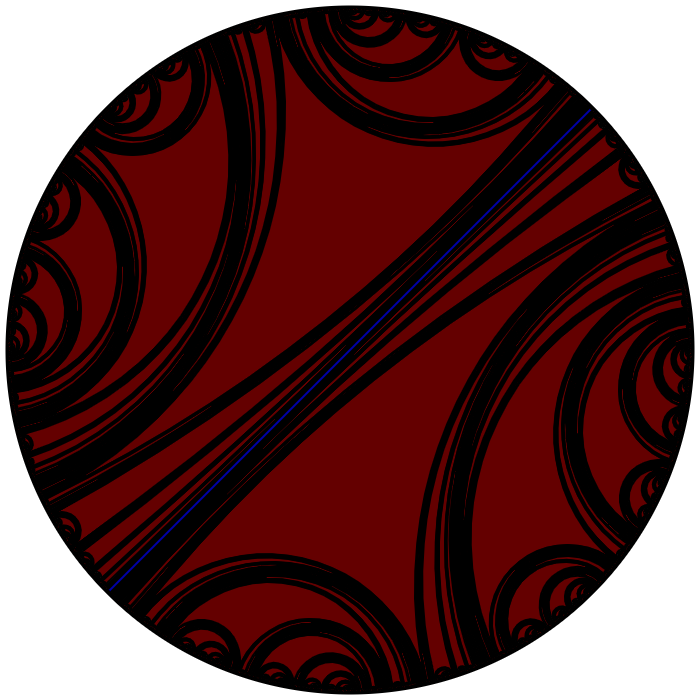}
	\captionof{figure}{The pullback lamination of the $\frac{1}{8}$ critical chord starting with the rabbit.}
	\label{critical_leaf-001.png}
\end{minipage}

\vspace*{1em}

\begin{ex}
	\cref{fignon-q,primitive} are the results of forcing the lamination to place a leaf at a certain point in the manner contemplated above. In \cref{fignon-q}, we force it to place a limit leaf at $0010\_001$. In the FDL of the sequence, there are already 2 other leaves there for a total of 3, making the lamination not clean. However, the lamination is not too bad and amounts to another way to approach the Misiurewicz point.

	\cref{primitive} is the result of forcing the lamination to have a leaf with the endpoint $\_0010$. If you have very good eyes you may be able to discern that the critical gap is a Fatou gap. This is an example of achieving a primitive bifurcation and forcing a periodic leaf to appear. It seems that at many points we can force there to be a leaf, but the lamination chooses what kind of leaf.
	\hfill
	\noindent

\end{ex}

\begin{minipage}[t]{0.48\textwidth}
	\vspace*{0cm}
	\includegraphics[width=\linewidth]{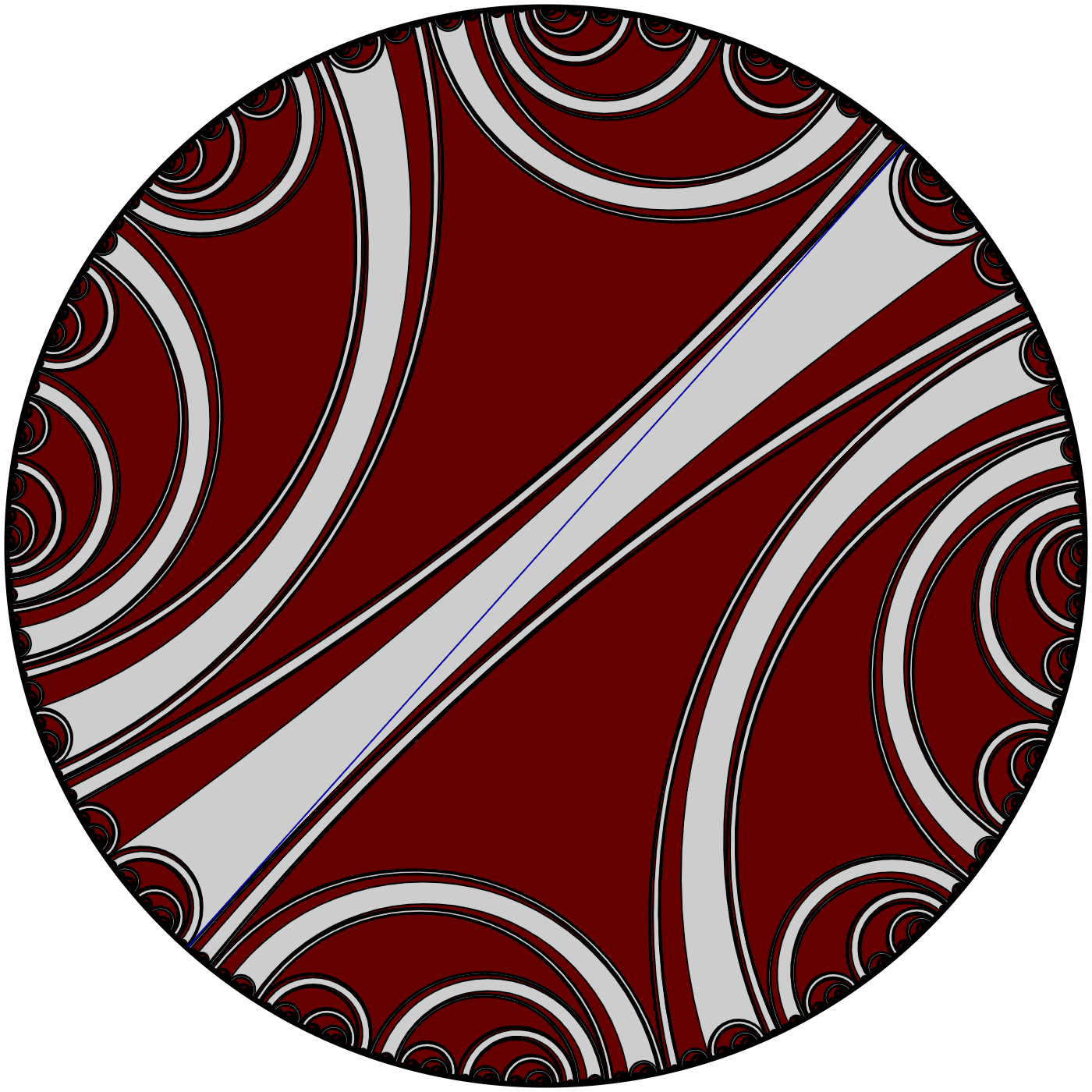}
	\captionof{figure}{An unclean pullback lamination.}
	\label{fignon-q}
\end{minipage}
\begin{minipage}[t]{0.48\textwidth}
	\vspace*{0cm}
	\includegraphics[width=\linewidth]{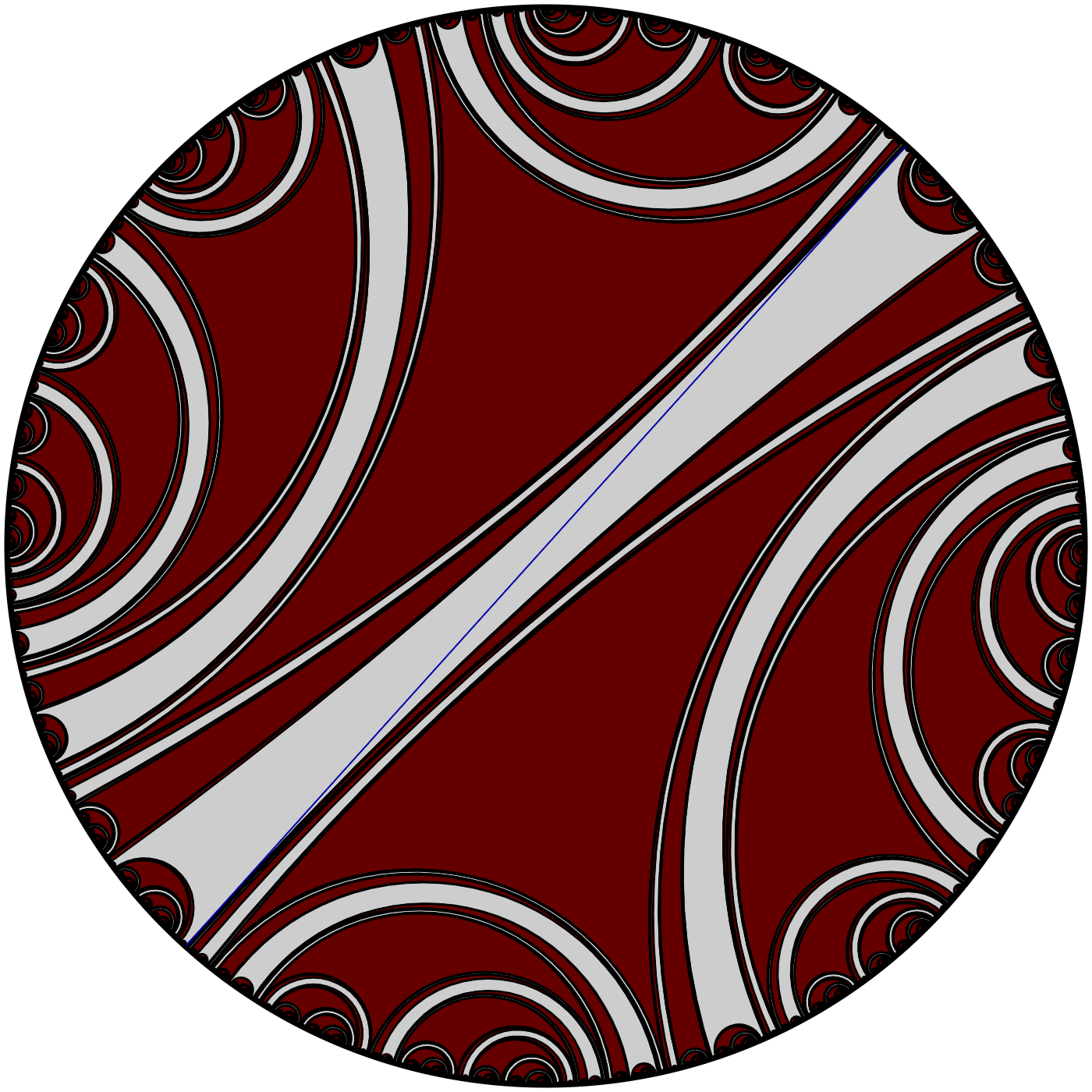}
	\captionof{figure}{A primitive bifurcation of the rabbit.}
	\label{primitive}
\end{minipage}

Consider whether all the equivalence classes are finite. It seems so from these examples, and the last one might give us a hint as to why. If we skip ahead to \cref{criticalsides} restrictions on the critical leaves show that there are no countable/caterpillar gaps according to a theorem in the 2022 version of \cite{bhattacharya2021unicriticallaminations}. So it may seem obvious that the equivalence classes are finite, but there might be more subtle ways to create infinite equivalence classes, and it is probably better to invoke the characterization meant for this question. The following definition and theorem are from \cite{blokh2012laminationslanguageleaves}.

\begin{definition}[Proper lamination]
	Two leaves with a common endpoint $v$ and the same image which is a
	leaf (and not a point) are said to form a \emph{critical wedge} (the
	point $v$ then is said to be its vertex). A lamination $\lam$ is
	\emph{proper} if it contains no critical leaf with periodic endpoint
	and no critical wedge with periodic vertex.
\end{definition}

\begin{thm}\label{t:nowander}
	Let $\lam$ be a proper
	invariant lamination. Then
	the equivalence relation of $\lam$ is an invariant laminational equivalence relation.
\end{thm}

Note that their definition of \emph{invariant laminational equivalence relation}, which we sometimes merely call \emph{laminational}, is similar to their notion of a q-lamination, and requires that classes are finite and map with positive orientation. We have an example where a limit leaf accumulates with polygons at both endpoints. There is no possible issue with this at preperiodic points. So it is important that all limit leaves have non-periodic endpoints. %

\begin{theorem}\label{proper}
	All leaves with a periodic endpoint in a lamination from the base of the pullback tree have both endpoints periodic of the same period. Thus, all laminations in the base of the tree are proper and sibling invariant.
\end{theorem}

\begin{proof}
	The fact that the lamination is invariant is \cref{siblinginvariant}. Since an FDL is a set of periodic and preperiodic leaves, we consider a leaf not appearing in any FDL, though we can not assume that no leaf from an FDL in the sequence shares an endpoint with it. We call the limit leaf $\overline{xv}$, where $v$ is the periodic point and $x\neq v$. Without loss of generality, we assume $v$ to be fixed, and we assume that the leaves approaching $\overline{xv}$ are on the CCW side of $v$ and the $CW$ side of $x$.

	Consider a leaf, $\overline{x_0v_0}$ from an FDL that is so close to $\overline{xv}$ that $d*(v_0 - v)<x_0-v$ and open arc from $(x_0,x]$ contains no fixed points. Thus, the image of the arc $(v,v_0)$ expands to include $v_0$, but not enough to include $x_0$. Thus, the next FDL in the sequence has preimage of $\overline{x_0v_0}$, call it $\overline{x_1v_1}$, such that $v_1$ is in the arc $(v,v_0)$. Since $x_1$ is not in the arc $(v,v_0)$, and $\overline{x_1v_1}$ does not cross $\overline{x_0v_0}$ or $\overline{xv}$, $x_1$ is in the arc $(x_0,x)$. Since $\overline{x_1v_1}$ is even closer, we can repeat this process indefinitely to form the sequence $\{\overline{x_iv_i}\}$. By construction $v_i \to v$, and by the monotone convergence theorem, $x_i \to y$. We also see that, $x_{i-1} = \sigma_d(x_{i})$ and $v_{i-1} = \sigma_d(v_{i})$. By the sequential criterion of $\sigma_d$ being a continuous function, $\sigma_d(y) =y$. But since we chose $x_0$ such that there is no fixed point in between $x$ and $x_0$, $y=x$ which is the only fixed point it is allowed to be.

	This argument can just as easily be applied to the CW side of $v$. If $v$ is of period $n$, the conclusion that $\sigma^n_d(y) =y$ would merely be an upper bound on the period of $y$. We can still establish $y=x$ since there are finitely many periodic points given a bound on the period. However, to see that the points are of the same period, one should reverse the argument and find that the period of $x$ is an upper bound of the period of $v$. With that said, the last statement of the theorem is the observation that a periodic leaf is not a critical leaf and a pair of periodic leaves is not a critical wedge.
\end{proof}

\section{Are all FDL realized by a polynomial?}\label{Realized}

Though the next logical step may be to contemplate loops in the closure of the pullback tree, a patient reader may be interested in whether all FDL have polynomials, and we are ready to answer that. After we take a moment to observe the absence of Siegel gaps, we can use Kiwi's theorem \cite{KIWI2004207}. Kiwi defined the term ``$\mathbb{R}$eal lamination'' to refer to laminational equivalence relations. \cite{blokh2012laminationslanguageleaves} states that forward invariance implies backward invariance. Otherwise, the stipulations in the definitions are in one to one correspondence with identical meaning. Thus, we use the term laminational as in \cite{blokh2012laminationslanguageleaves}.

\begin{theorem} %
	An equivalence relation, $\lambda$, of the circle is (in the sense of impressions) the lamination of a polynomial $f$ with
	connected Julia set and without irrationally neutral cycles if and only if $\lambda$ is a laminational equivalence relation without rotation curves.
\end{theorem}

\begin{lemma} \label{criticalsides}
	If a gap of a lamination in the base of the pullback tree has a critical leaf in its boundary, then each leaf in the boundary of the equivalence class is critical.
\end{lemma}
\begin{proof}
	Suppose we have the critical leaf $\overline{ab}$. This leaf is the limit of a sequence of leaves from at least one side. We take a sequence of leaves converging from only one side and call it $\overline{a_ib_i}$. Since the side does not really matter, we suppose $a<a_i<b_i<b$. We choose a leaf so close that $a_i-a<\epsilon$ and $b-b_i<\epsilon$. We call the sibling of $a_i$ CCW of $b$ $c_i$, and the CCW endpoint of the sibling of $\overline{a_ib_i}$ that has $c_i$ as its CW endpoint $d_i$. Since $c_i$ is within $\epsilon$ of $b$ and $d_i$ is within $\epsilon$ of some sibling of $a$ and $b$. Form a sequence of such siblings. Since $a$ and $b$ have finitely many siblings, we find that the sequence must accumulate on some critical leaf with a CW endpoint at $b$. Perhaps we have found a sequence that approaches $\overline{ab}$ from the other side, or perhaps we have found that the initial critical leaf must have another to its CCW that touches it. In the latter case, we can find a sequence of them until there is no more room in the circle, and we return to our starting point, $a$.
\end{proof}

\begin{theorem}\label{realized}
	All laminations in the base of the pullback tree satisfy the hypothesis of Kiwi's realization theorem. Moreover, for every FDL, there is a polynomial with a connected Julia set such that for each non-singular equivalence class of the FDL, there is a point in the Julia set where exactly the angles of that equivalence class have rays landing at it.
\end{theorem}

\begin{proof}
	By \cref{proper} and \cref{t:nowander}, we have an invariant laminational equivalence relation. The only remaining condition is NR, which stands for ``no rotation curves''. A rotation curve is a periodic, simple closed curve, in the quotient of the equivalence relation with the circle such that the first return is a homeomorphism. Since a simple closed curve can only pass through any cut point of the quotient once, the curve corresponds to a gap of the lamination. Since a curve is uncountable, it must correspond to an uncountable gap of the lamination. Given that the gap is periodic with degree 1, it is classified by \cite{blokh2015combinatorialmandelbrotsetquotient} as a Siegel gap, and it is proven to contain a critical leaf in the boundary of its iterate. By \cref{criticalsides} no such gap exists. Thus, the lamination is realized by a polynomial with a connected Julia set.

	Kiwi's theorem does not show that all rays land in the way described by the lamination but instead that the impression of each ray intersects as described by the lamination. Along with the fact that all FDL have at least one child, this is enough to establish that the co-landing locus of any FDL is non-empty. The vertices of polygons in an FDL are all periodic or preperiodic, which is the same as rational. Using the connectedness of the Julia set, we can see from \cite{Orsay} that the rational rays land. Also from \cite{Orsay}, we know all rational rays land at periodic or preperiodic point in the Julia set. Kiwi makes it clear with his characterization of what it means for a ray to land that the landing point of a ray is in its impression. Thus, his corollary 1.2 applies, and the impression of the ray is its landing point.
\end{proof}

There is one more statement required for an adequate discussion of the polynomials represented.

\begin{lemma}\label{childlemma}
	Given any polynomial, $p$, and an FDL such that the former is in the
	co-landing locus of the latter, some child of the FDL contains $p$ in its
	co-landing locus.
	Alternately, given any laminational equivalence relation, $\lam$, and a FDL $\lam_1$ such that $\lam_1$ is finer than $\lam$, there exists an FDL, $\lam_2$ that is a child of $\lam_1$ and $\lam_2$ is also finer than $\lam$.
\end{lemma}
\begin{proof}
	The situation of the polynomial is the same as the situation of the equivalence class partly because all the angles of the vertices in the FDL are rational. Even if some rays do not land for $p$, all rational rays must land. Since rays do not cross, the case of the polynomial is the same as of that of the lamination.

	In the lamination, each equivalence class has a certain number of preimages and those  preimages are unlinked. Pulling back all polygons in accordance with the laminational relation cannot fail to be an FDL. Any finite subset of a lamination is a lamination, and for each of the seven conditions to make a lamination FDL, there is a reason why it is satisfied.

\end{proof}

\section{Loops in the closure of the pullback tree?}\label{Loops}

First let us name two phenomena that may be confused with a loop in the closure of the pullback tree that should not be regarded as such.
\begin{enumerate}
	\item It is obvious from the abundance with which these sequences add periodic leaves that sequences to the end of two different trees can arrive at the same place. One example of this is in degree 3 the tree with the root containing the rotational polygons of \_001 and \_112. At least conjecturally, the base of that tree is equal to the base of the tree having the root containing those polygons and additionally, a leaf from \_0 to \_1. While this phenomenon is interesting, it does not illuminate the structure imposed on parameter space by a single pullback tree.
	\item There is presumably an abundance of sequences in the base of the pullback tree that have the same limit. This phenomenon is not particularly useful in polynomial parameter space: although we can convert those sequences into sequences of FDL, they would not be sequences to the end of the tree, and thus there would be no particular relationship between the co-landing loci of those FDL.
\end{enumerate}

Thus, what we mean by a ``loop in the closure of the pullback tree'' is two sequences to the end of the same pullback tree. In such sequences, the co-landing locus of any FDL is a subset of that of the previous. Consider the first pair of non-equal laminations from some index of both sequences. The co-landing loci of those two laminations should intersect, and the intersection should include a polynomial whose lamination is the limit of those two sequences. Such intersection will create a graph structure in parameter space that we model with a graph of that generation of the tree. We find that these correspond to \emph{critical polygons}, which are polygons that map with degree greater than one.

\begin{prop}
	If two different sequences to the end of the same pullback tree limit to laminations with the same equivalence relation, then the q-lamination of that equivalence relation contains a critical polygon. %
\end{prop}
\begin{proof}
	At some point the two sequences of laminations must have non-equal sibling portraits, yet those two non-equal FDL both have equivalence relations finer than the end equivalence relation.

\end{proof}

We saw something like this in \cref{fignon-q}, where one way of seeing it is as a limit to the Misiurewicz point. There is one way to reach the point such that the end lamination is a q-lamination. Alternately, we can start with an FDL that contains the point in the boundary of its co-landing locus and then at each step of the sequences, choose the child FDL that has that point in the boundary. Perhaps the best way to phrase it is in terms of co-landing loci.

\begin{prop}
	Give two FDL on the same level of the pullback tree, they have intersecting co-landing loci iff there is an FDL on that level that is courser than either of them.
\end{prop}
\begin{proof}
	The same reasoning applies as for previous proposition. Mainly this follows from \cref{childlemma}.
\end{proof}

This justifies an interest in the finer relation on the FDL on the same level of the pullback tree. The finer relation might be a bit messy in higher degree, so make things easier to draw, we add a notion, which hopefully adds information.
\begin{definition}[trapped criticality]
	The \emph{trapped criticality} of an FDL is the total degree of all the polygons minus the number of polygons.
\end{definition}

\begin{definition}
	The \emph{free criticality} of an FDL is the number of critical polygons, with multiplicity, that fit outside the polygons of the FDL.
\end{definition}

By \cref{freecriticality}, the free criticality and the trapped criticality add to $d-1$.

\begin{conj}
	The complex dimension of the co-landing locus in the set of affine conjugacy classes of degree d polynomials is equal to the free criticality of the FDL.
\end{conj}

\begin{definition}[generational FDL graph]
	The \emph{generational FDL graph} is a directed graph whose vertices are all the FDL from one level of one pullback tree such that there is an edge from $a$ to $b$ iff
	\begin{enumerate}
		\item the trapped criticality of $b$ is one greater than that of $a$ and
		\item the equivalence relation $a$ is finer than the equivalence relation of $b$.
	\end{enumerate}
\end{definition}

See \cref{5thgeneration}.

\begin{cor}
	The transitive closure of the generational FDL graph is the same as the proper subset relation on the co-landing loci of the FDL in the graph.
\end{cor}
\begin{proof}
	Follows from the definitions and \cref{realized}.
\end{proof}

\begin{figure}[ht]
	\includegraphics[width=0.48\linewidth]{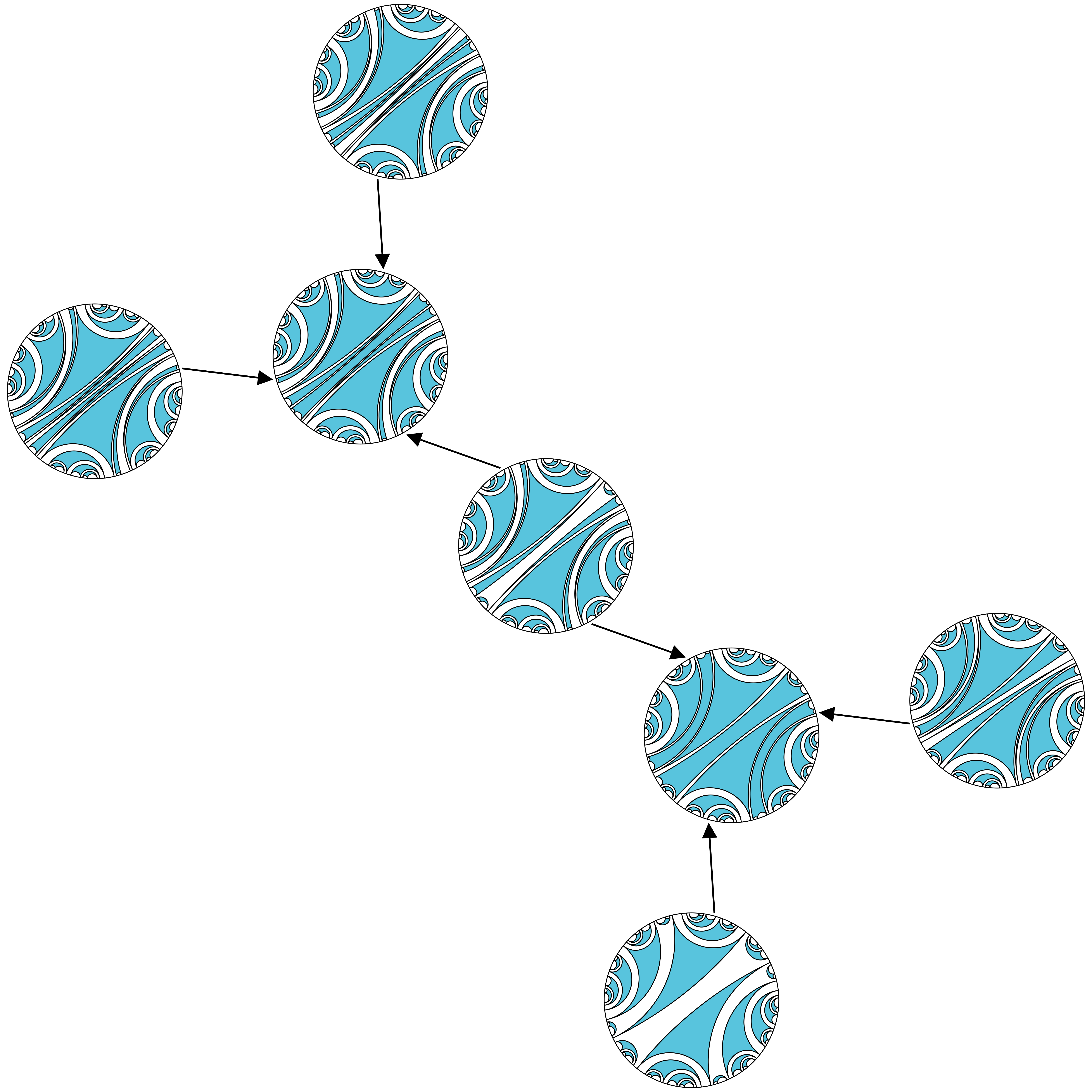}
	\includegraphics[width=0.48\linewidth]{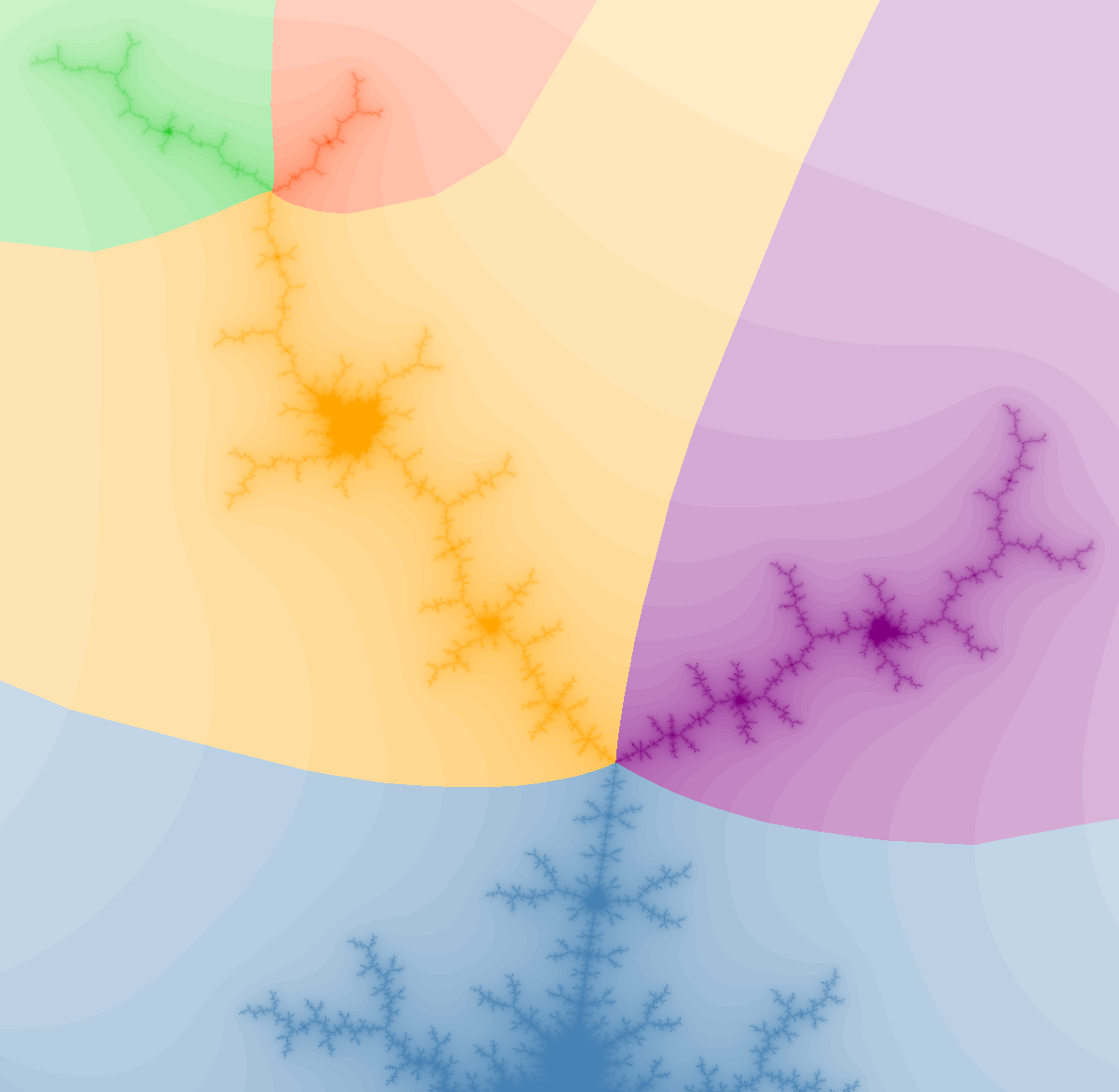} \\
	\caption{Pictured is the 5th generation of the quadratic rabbit, and a coloring of the co-landing loci.}
	\label{5thgeneration}
\end{figure}

\section{Future Directions of Research}

We list a number of questions and conjectures approximately ordered by how combinatorial they are. We start with the most combinatorial topics and end with the most analytic.
\begin{enumerate}
	\item The number of FDL on level $n$ for any pullback tree forms a sequence that, to the best of the author's knowledge, does not appear in The On-Line Encyclopedia of Integer Sequences (OEIS), with one exception.
	      \begin{conj}
		      Consider the quadratic pullback tree starting with the lamination containing only the leaf from $\_01$ to $\_10$. The number of FDL on each level of the tree is A152046 in \cite{OEIS}.
	      \end{conj}

	      Anyone attempting to work on this should look to \cite{ManimHilton} to for instructions on obtaining pictures of other pullback trees. Moreover, you should be warned of two things in any pullback tree, FDL are varied in their number of children and their children are varied in how many subsequent generations until the next branch point. One observation in the quadratic case is that pulling back short gives the longest delay between branch points, and it also increases the order of the next branch point. The reason is that it puts multiple polygons in the image of the critical round gap.

	\item
	      \begin{conj}
		      The co-landing locus of any FDL intersected with the connectedness locus is closed and perhaps connected.
	      \end{conj}

	      Regarding the closed part, Proposition 5.1 \cite{schleicher1998} states that the landing point of a rational ray of a quadratic polynomial will very continuously with the parameter. Hopefully, this can be generalized.

\end{enumerate}

\section{Appendix}

The following proof is vestigial and perhaps meritless. It has been provided as a demonstration that the notion of canonical pullback laminations can be developed without the use of critical chords.

\begin{proposition}\label{vestigial}
	Suppose there is an FDL, $L$, where all round gaps have a degree and at least one round gap has a degree of more than one. Then there exists a lamination in the base of $L$'s branch of the pullback tree with a hyperbolic gap inside each of the round gaps of $L$ such that the hyperbolic gap has the same degree as the round gap.
\end{proposition}

\begin{proof}
	We form a sequence to the end of the pullback tree, $L \in \{L_i\}$, by describing the sibling portrait in each round gap, but first we must grasp all the round gaps involved. $L_0$ is the set of periodic leaves in $L$. Each round gap of $L_0$ is either degree 1 or partly critical, here partly critical means a round gap whose image contains the circle, but does not map as a cover. Thus, each round gap of $L_0$ is either partly critical or carries homeomorphically onto a partly critical round gap. Each critical gap of $L_{i>0}$ must be a subset of a partly critical round gap of $L_0$. By \cref{freecriticality}, all round gaps of $L_{i>0}$ have a degree. Thus, if $R$ is a round gap with degree $d$ of $L_{i+1}$, then $R$ maps as a degree $d$ covering map onto a gap of $L_i$.

	Let $L=L_n$, and fix $i\geq n$. Every round gap of $L_i$ is either critical, homeomorphically carries onto a critical round gap, onto a partly critical gap of $L_0$. Considering the last case, the partly critical round gap of $L_0$ may contain a critical round gap of $L$, or it may contain only critical polygons of $L$. In the second case the round gaps of $L$ in the partly critical gap are degree 1. But since $\sigma_d$ is expanding and round gaps contain circle arcs on their boundary, eventually some iterate of any round gap will contain a critical round gap of $L$. Thus starting with any degree 1 round gap of $L$, some subset of that gap will carry homeomorphically onto a critical round gap of $L$. The leaves on the boundary of this subset are compatible with all laminations $L_i$ because they are pullbacks of the leaves of critical polygons of $L$, and they are forced to be pulled back in that way by the chained homeomorphisms.

	Form a function, $f_i$, from the set of critical round gaps of $L_i$ into the set of critical round gaps of $L_i$ such that $\sigma_d(R)$ has a subset that is equal to or carries homeomorphically onto $f(R)$. Form $L_{i+1}$ as follows: if a round gap has degree 1, then there is no choice. If a round gap, $R$ has a degree, $d$, then choose the round gap in its image that contains a subset that carries homeomorphically onto $f(R)$, and then pull back that gap $d$ to 1 into $R$. This process puts a round gap of the same degree inside each round gap. For consistency, we form $f_{i+1}$ based by substituting each round gap with the critical round gap inside of it in $L_{i+1}$. It is easy enough to see that this process is continuable and forms a sequence to the end of the pullback tree. We attempt to discern that the limit has the desired properties.

	Form a sequence, $\{R_i\}$ starting with a critical round gap of $L=L_n$ that is cyclic under $f_n$. Find each subsequent term of the sequence and at each term taking the critical round gap found inside of it. Since the change from $f_i$ to $f_{i+1}$ merely narrows each element of the domain and range, there is some number of steps $m$ such that $\forall i, \sigma^m_d(R_{i+m}) = R_i$, and $\sigma^m_d$ maps $R_i$ forward with some fixed degree $w$. Since the gaps are nested, they converge to a continuum in the disk. Its boundary must be the limit of the boundary's of the round gaps. Take a point from a circle arc of $R_i$. Take its $w$ preimages in the boundary of $R_{i+m}$. Note that since $\sigma_d^m$ maps these gaps as a covering map, the preimages are evenly spaced in the boundary of $R_{i+m}$. As we do this repeatedly, we find that the continuum is a gap of the end lamination having uncountably many points of the circle. Since all the original critical round gaps are cyclic under $f_n$, it is clear that they also become hyperbolic.

	If it does not seem obvious that this hyperbolic gap is critical, consider that each of the round gaps in the sequence is critical, say of degree $x$. And since the boundaries limit to the boundary, for all leaf or degenerate leaf in the boundary of the end gap, and for all $\epsilon$ there exists a leaf or degenerate leaf in the boundary a round gap within $\epsilon$. That point has siblings in the round gap. Thus, we can form a sequence of sets of $x$ sibling leaves in round gaps. It is clear from the increasing number of leaves in the boundary of the round gaps that these sets are of disjoint, far-apart leaves. The sequence of images of these leaves and degenerate leaves converges to the image of the chosen leaf or degenerate leaf in the boundary of the end gap. Since the preimage function of $\sigma_d$ is continuous, and since these sets of leaves and degenerate leaves remain far apart, the chosen leaf or degenerate leaf in the boundary of the end gap has a collection of $x$ leaves or degenerate leaves in the boundary of that gap.

\end{proof}
\bibliographystyle{ieeetran}
\bibliography{IEEEabrv,./fdl.bib}
\end{document}